%% file: main.tex
\documentclass[11pt]{article}

\usepackage[margin=1in]{geometry}
\usepackage{amsmath,amssymb,amsthm}
\usepackage{mathtools}
\usepackage{booktabs}
\usepackage{graphicx}
\usepackage{tikz}
\usetikzlibrary{arrows.meta,calc,positioning,matrix,fit}
\usepackage[T1]{fontenc}
\usepackage{lmodern}
\usepackage{microtype}
\usepackage[numbers,sort&compress]{natbib}
\usepackage{hyperref}
\hypersetup{
  pdftitle={A Conditional Rank-Count Theory for the Combinatorial Discretizable Distance Geometry Problem},
  pdfauthor={Michael Souza, Wagner da Rocha, Carlile Lavor},
  colorlinks=true,
  linkcolor=blue,
  citecolor=blue,
  urlcolor=blue
}

\theoremstyle{plain}
\newtheorem{theorem}{Theorem}[section]
\newtheorem{lemma}[theorem]{Lemma}
\newtheorem{proposition}[theorem]{Proposition}

\theoremstyle{definition}
\newtheorem{definition}[theorem]{Definition}
\newtheorem{assumption}[theorem]{Assumption}

\DeclareMathOperator{\PredCl}{PredCl}
\DeclareMathOperator{\rank}{rank}
\DeclareMathOperator{\supp}{supp}

\title{A Conditional Rank-Count Theory for the Combinatorial Discretizable Distance Geometry Problem}
\author{
Michael Souza\\
Universidade Federal do Cear\'a\\
\texttt{michael@ufc.br}
\and
Wagner da Rocha\\
Unicamp\\
\texttt{wdarocha@ime.unicamp.br}
\and
Carlile Lavor\\
Unicamp\\
\texttt{clavor@unicamp.br}
}
\date{}

\begin{document}

\maketitle

\begin{abstract}
The Combinatorial Discretizable Distance Geometry Problem combines a finite
binary lateration process with additional distance constraints. When
predecessor sets are not consecutive, these pruning constraints interact in
ways that make the symmetry-based counting methods available for molecular
instances insufficient. We establish an exact, dimension-uniform counting theorem for seed-fixed
feasible realizations under strict discretization and a generic feasible
framework assumption. Our approach identifies partial reflections through
seed-free connected components of the lateration skeleton. These reflections
are encoded by binary masks, while a labeled constraint matrix detects
combinations that preserve all pruning distances. The feasible branch choices
split into constrained choices generated by compatible partial reflections and
unconstrained choices outside the predecessor closure of the pruning
endpoints. The resulting count is obtained from graph operations and rank computations
over the binary field, without enumerating the lateration tree. Completeness
follows from the characterization of generic realizations of the relevant
joined graphs by partial reflections. The theorem applies in every Euclidean
dimension, including dimension one.
\end{abstract}



\input{sections/introduction}
\input{sections/related_works}
\input{sections/subproblems}
\input{sections/generators}
\input{sections/violations}
\input{sections/results}
\input{sections/example}
\input{sections/conclusion}

\bibliographystyle{plainnat}
\bibliography{references}

\end{document}

%% file: sections/introduction.tex
\section{Introduction}
\label{sec:intro}

Distance geometry asks for a realization of an edge-weighted graph in a fixed
Euclidean dimension. Given a simple graph \(G=(V,E)\), positive edge lengths
\(d:E\to\mathbb R_{>0}\), and \(K\ge1\), the goal is to find
\(x:V\to\mathbb R^K\) satisfying
\[
  \|x_u-x_v\|=d(\{u,v\})
  \qquad (\{u,v\}\in E).
\]
The problem underlies molecular conformation, sensor-network localization,
nanostructure reconstruction, and several inverse problems in data analysis
\cite{liberti2008branch,yemini1978positioning,mucherino2020analysis,
liberti2020distance}.

A prescribed vertex order can turn this continuous feasibility problem into a
finite one. In a discretizable order, an affinely independent seed of \(K\)
vertices is fixed, and each later vertex has distances to \(K\) earlier
predecessors. When those predecessors form a metric clique, their distances
determine an affine hyperplane, and the next vertex has two reflected
positions. Repeating this construction produces a binary lateration tree
\cite{mucherino2012discretizable,lavor2012discretizable}. The resulting class
is the Combinatorial Discretizable Distance Geometry Problem (CDDGP). The
remaining graph edges, called pruning edges, discard branch codes whose
realizations violate an additional distance.

For the molecular specialization, in which every predecessor set consists of
the \(K\) immediately preceding vertices, partial reflections act on suffixes
of the order. This nested structure explains the familiar power-of-two counts
and supports efficient symmetry-aware algorithms
\cite{mucherino2012exploiting,liberti2013counting,liberti2014number}. In the
CDDGP, predecessor sets need not be consecutive. A branch reflection can then
propagate through a noninterval set of descendants, and different reflection
mirrors can produce the same binary mask while interacting differently with
pruning edges. This is precisely where the suffix theory breaks down.

The present paper gives a graph-theoretic and linear-algebraic resolution. The
central construction removes a predecessor clique \(C\) from the constrained
lateration skeleton and uses its seed-free connected components as the atomic
sets that may be reflected across \(\operatorname{aff}(C)\). Each such
component defines a labeled binary mask. The mask matrix \(\mathsf M\) maps
selected component reflections to branch-code shifts, while a labeled
violation matrix \(\mathsf V\) records whether a pruning edge crosses a
component with a fixed endpoint outside its mirror clique. The binary space
\[
  \mathcal K=\mathsf M(\ker\mathsf V)
\]
therefore contains exactly the shifts certified by compatible partial
reflections.

Our contributions are as follows.
\begin{enumerate}
  \item We isolate the predecessor-closed subsystem containing every pruning
        endpoint and prove an exact product decomposition into constrained and
        free branch bits.
  \item We introduce component-reflection masks and a mirror-labeled violation
        matrix, and prove \emph{global stabilizer soundness}: every shift in
        \(\mathcal K\) preserves every pruning distance for every branch code
        on the strict-discretization parameter domain.
  \item We prove generic completeness in every dimension \(K\ge1\). The
        constrained lateration skeleton is a \(K\)-tree; after pruning edges
        are added, the graph is \(K\)-joined. A theorem of
        Garamv\"olgyi and Jord\'an then expresses every equivalent generic
        realization by partial \(K\)-reflections, whose branch masks we show
        belong to \(\mathcal K\).
  \item We derive the rank-count formula
        \[
          |\mathcal X|
          =2^{|B_F|+\rank_{\mathbb F_2}
          \left[\begin{smallmatrix}\mathsf M\\\mathsf V\end{smallmatrix}\right]
          -\rank_{\mathbb F_2}(\mathsf V)},
        \]
        where \(\mathcal X\) denotes seed-fixed feasible branch codes.
\end{enumerate}

The result is conditional only in the natural geometric sense required by the
partial-reflection theorem: the instance must satisfy strict discretization
and its constrained graph must admit a feasible realization with a generic
representative before seed normalization.

The paper is organized as follows. Section~\ref{sec:related} reviews the
counting and rigidity background. Section~\ref{sec:subproblems} defines the
CDDGP model, branch codes, and the constrained/free decomposition.
Sections~\ref{sec:generators} and~\ref{sec:violations} construct the masks and
prove soundness. Section~\ref{sec:results} establishes generic completeness
and the rank-count theorem. Section~\ref{sec:example} gives a compact worked
example and describes the accompanying exhaustive validation. Conclusions are
drawn in Section~\ref{sec:conclusion}.

%% file: sections/related_works.tex
\section{Related Work}
\label{sec:related}

\subsection{Discretization and counting}

The Branch-and-Prune paradigm constructs realizations by sequential
lateration and rejects a partial realization as soon as an available distance
constraint fails \cite{liberti2008branch,mucherino2012discretizable}. Its
performance is strongly influenced by the symmetry of the discretization
order. In the DMDGP, where the predecessors of vertex \(v_i\) are
\(v_{i-K},\ldots,v_{i-1}\), the two lateration choices are related by a
reflection that acts on the suffix beginning at \(v_i\). These suffix
reflections generate a binary symmetry group and lead to generic
power-of-two solution counts
\cite{mucherino2012exploiting,liberti2011number,liberti2013counting,
liberti2014number}. Particularly close to the present viewpoint is the
symmetry-based build-up algorithm of Gon\c{c}alves et
al.~\cite{goncalves2021new}. It decomposes a DMDGP instance into nested
subproblems indexed by pruning edges and selects compositions of partial
reflections that satisfy each new distance constraint. Its purpose is to find
a first realization efficiently, whereas our purpose is to count all generic
feasible codes; moreover, the CDDGP requires labeled component masks in place
of the nested suffix reflections available in the DMDGP.

The broader DDGP allows arbitrary earlier predecessor sets. Suitable orders
can be easier to recognize than contiguous DMDGP orders
\cite{cassioli2015discretization,abud2018k}, but the realization count can
depend on numerical edge lengths. Abud et al.~\cite{abud2024impossible} prove
that a weight-independent counting method is impossible for general DDGP
instances. Their positive subclass requires each predecessor set to induce a
clique; this is the CDDGP convention adopted here. The clique condition fixes
the predecessor metric before the next lateration step, but it does not restore
the nested suffix supports that make DMDGP counting immediate.

\subsection{Partial reflections and rigidity}

Partial reflections are a classical source of noncongruent equivalent
frameworks. Reflecting a union of components separated by a \(K\)-vertex set
across the affine hull of that separator preserves every edge whose endpoints
are moved together or whose fixed endpoint lies in the separator. In
discretized distance geometry, the same operation is the geometric mechanism
behind branch-code symmetries.

Our completeness proof uses the recent theory of \(d\)-joined graphs developed
by Garamv\"olgyi and Jord\'an \cite{garamvolgyi2024partial}. They prove that a
generic realization of a \(d\)-joined graph generates all equivalent
realizations, up to congruence, by reduced sequences of partial
\(d\)-reflections. They also show that every \(d\)-connected chordal graph is
\(d\)-joined and that edge addition preserves \(d\)-joinedness. These results
are particularly well matched to CDDGP: a predecessor-closed lateration
skeleton is a \(K\)-tree, hence chordal and \(K\)-connected, while pruning
constraints are exactly edge additions.

The role of rigidity theory here differs from global-rigidity certification.
We do not seek to eliminate all reflection ambiguities; those ambiguities are
the objects being counted. Instead, the \(K\)-joined characterization provides
a complete geometric generating set, and the binary matrices introduced below
identify the combinations compatible with the pruning edges.

\subsection{Algebraic representation}

Binary encodings of Branch-and-Prune paths and the action of DMDGP partial
reflections by addition in \(\mathbb F_2^n\) were developed by Liberti et
al.~\cite{liberti2014number}; the degrees of freedom of binary
representations for discretizable distance geometry were subsequently studied
by Mucherino~\cite{mucherino2020analysis}. These constructions exploit the
nested suffix structure of consecutive predecessor orders. In the broader
CDDGP setting, however, a binary mask alone does not identify the affine
mirror that realizes it: equal masks may arise from different separator
cliques and may have different effects on a pruning edge. Our labeled columns
therefore retain both the mask and its mirror clique. This is the essential
distinction from an unlabeled parity model: cancellations are meaningful only
within a common mirror label. The resulting construction turns geometric
compatibility into kernel and image computations over \(\mathbb F_2\), while
the \(K\)-joined theorem supplies the converse needed for exact counting.

%% file: sections/subproblems.tex
\section{CDDGP Model and Branch Codes}
\label{sec:subproblems}

Let the vertex order be \(v_1,\ldots,v_n\), and fix the seed
\(V_0=\{v_1,\ldots,v_K\}\) at affinely independent coordinates in
\(\mathbb R^K\). For each \(i>K\), let
\[
  C_i\subseteq\{v_1,\ldots,v_{i-1}\},
  \qquad |C_i|=K,
\]
be the prescribed predecessor set of \(v_i\). Define
\[
  E_0=E[V_0],
  \qquad
  E_D=\{\{u,v_i\}:i>K,\ u\in C_i\},
  \qquad
  G_D=(V,E_0\cup E_D).
\]
The remaining edges are the pruning edges
\[
  E_P=E\setminus(E_0\cup E_D),
  \qquad E=E_0\mathbin{\dot\cup}E_D\mathbin{\dot\cup}E_P.
\]

\begin{definition}[Skeletal CDDGP template]
\label{def:cddgp-template}
The ordered template is a \emph{skeletal CDDGP template} if
\(G_D[C_i]\) is a \(K\)-clique for every \(i>K\). Thus the complete metric
of each predecessor set is already present in the lateration skeleton, before
pruning edges are imposed.
\end{definition}

This distinction is important. Requiring \(C_i\) to be a clique only in the
full graph would allow a pruning edge to define part of the predecessor
metric, coupling branch creation and branch rejection. The skeletal convention
keeps these roles separate.

\begin{assumption}[Strict discretization (SD)]
\label{ass:sd}
For each \(i>K\), whenever the metric clique \(C_i\) is realized with its
prescribed distances, the \(K\) spheres centered at its vertices, with radii
\(d(\{u,v_i\})\), \(u\in C_i\), intersect in exactly two distinct points.
\end{assumption}

The two points in Assumption~\ref{ass:sd} are reflected across the hyperplane
spanned by the predecessor clique. The assumption is uniform over branch
choices: every prescribed predecessor metric produces two nondegenerate
candidates.

Let \(B=V\setminus V_0\) be the branch-vertex set. Choose an ordering
\((c_{i,1},\ldots,c_{i,K})\) of each \(C_i\). For a nondegenerate
lateration embedding \(x\), define
\[
  \Delta_i(x)=
  \det\bigl[
    x_{c_{i,2}}-x_{c_{i,1}},\ldots,
    x_{c_{i,K}}-x_{c_{i,1}},
    x_{v_i}-x_{c_{i,1}}
  \bigr].
\]
For \(K=1\), this is the signed displacement from the unique predecessor.
The branch bit \(b_i(x)\in\mathbb F_2\) records the sign of \(\Delta_i(x)\).

\begin{lemma}[Branch-code parametrization]
\label{lem:branch-parametrization}
Under the skeletal CDDGP convention and Assumption~\ref{ass:sd}, the branch
map is a bijection
\[
  \Phi:\mathbb F_2^B\longrightarrow
  \{\text{lateration embeddings of }G_D\text{ with }V_0\text{ fixed}\}.
\]
\end{lemma}

\begin{proof}
Proceed in the vertex order. Once \(C_i\) has been placed, its prescribed
metric is fixed because \(G_D[C_i]\) is a clique. Assumption~\ref{ass:sd}
gives two distinct candidates for \(v_i\), and the sign bit selects one
uniquely. This constructs an embedding for every code. If two codes first
differ at \(v_i\), they select distinct candidates from the same predecessor
placement, so the resulting embeddings differ. Conversely, the oriented-volume
signs recover the code of every lateration embedding.
\end{proof}

Define the feasible branch-code set
\[
  \mathcal X=
  \{s\in\mathbb F_2^B:
  \|\Phi_a(s)-\Phi_b(s)\|=d(\{a,b\})
  \text{ for all }\{a,b\}\in E_P\}.
\]
Counts in this paper are seed-fixed branch-code counts. We do not quotient
again by Euclidean congruence.

\subsection{Pruning closure and free bits}

For \(S\subseteq V\), let \(\PredCl(S)\) be the smallest set
\(T\supseteq S\) such that
\[
  v_i\in T,\ i>K \quad\Longrightarrow\quad C_i\subseteq T.
\]
Set
\[
  V_P=V_0\cup\PredCl(V(E_P)),
  \qquad B_P=V_P\setminus V_0,
  \qquad B_F=B\setminus B_P.
\]
The coordinates of every pruning endpoint depend only on the branch bits in
\(B_P\). Let \(\mathcal X_P\subseteq\mathbb F_2^{B_P}\) be the constrained
codes whose induced lateration embedding on \(V_P\) satisfies all pruning
edges. When \(E_P=\varnothing\), we use
\(V_P=V_0\), \(B_P=\varnothing\), and \(\mathcal X_P=\{0\}\).

\begin{proposition}[Free-bit factorization]
\label{prop:free-bits}
Under Assumption~\ref{ass:sd}, restriction to constrained and free bits gives
a bijection
\[
  \mathcal X\cong\mathcal X_P\times\mathbb F_2^{B_F}.
\]
Consequently,
\[
  |\mathcal X|=2^{|B_F|}|\mathcal X_P|.
\]
\end{proposition}

\begin{proof}
Every pruning endpoint and every predecessor needed to place it lies in the
predecessor-closed set \(V_P\), so pruning feasibility depends only on
\(B_P\). Fix \(s_P\in\mathcal X_P\) and any free assignment
\(s_F\in\mathbb F_2^{B_F}\). Construct the vertices in \(B_F\) in the
inherited order. Their predecessors have already been placed, their clique
metric is prescribed, and Assumption~\ref{ass:sd} supplies two candidates;
the requested free bit selects one. This gives a unique full extension. No
pruning endpoint moves, so feasibility is preserved. The inverse map is
restriction to \(B_P\) and \(B_F\).
\end{proof}

%% file: sections/generators.tex
\section{Component Reflections and Binary Masks}
\label{sec:generators}

The constrained subsystem is the predecessor-closed graph on \(V_P\). Its
relevant mirror cliques are
\[
  \mathcal C_P=\{V_0\}\cup\{C_i:v_i\in B_P\}.
\]
Fix \(C\in\mathcal C_P\) and remove it from the constrained lateration
skeleton:
\[
  G_C=\bigl(V_P\setminus C,E_D[V_P\setminus C]\bigr).
\]
Write
\[
  \operatorname{cc}(G_C)
  :=\{A\subseteq V_P\setminus C:
    A\text{ is the vertex set of a connected component of }G_C\}.
\]

\begin{definition}[Relevant mirror components]
\label{def:mirror-components}
The relevant components for mirror \(C\) are
\[
  \widehat G_C=
  \{A\in\operatorname{cc}(G_C):
    A\cap V_0=\varnothing,\ B_P\cap A\ne\varnothing\}.
\]
Seed-containing components are anchored, whereas components disjoint from
\(B_P\) induce the zero constrained mask.
\end{definition}

Figure~\ref{fig:component-reflection} illustrates the geometric operation
behind this definition. The component \(A_1\) is reflected across the affine
span of the separator \(C\), while the seed and the other component \(A_2\)
remain fixed. An edge with both endpoints on the same side, or with its fixed
endpoint in \(C\), preserves its length; a crossing pruning edge with its
fixed endpoint outside \(C\) is a violation.

\begin{figure}[t]
  \centering
  \includegraphics[width=0.4\textwidth]{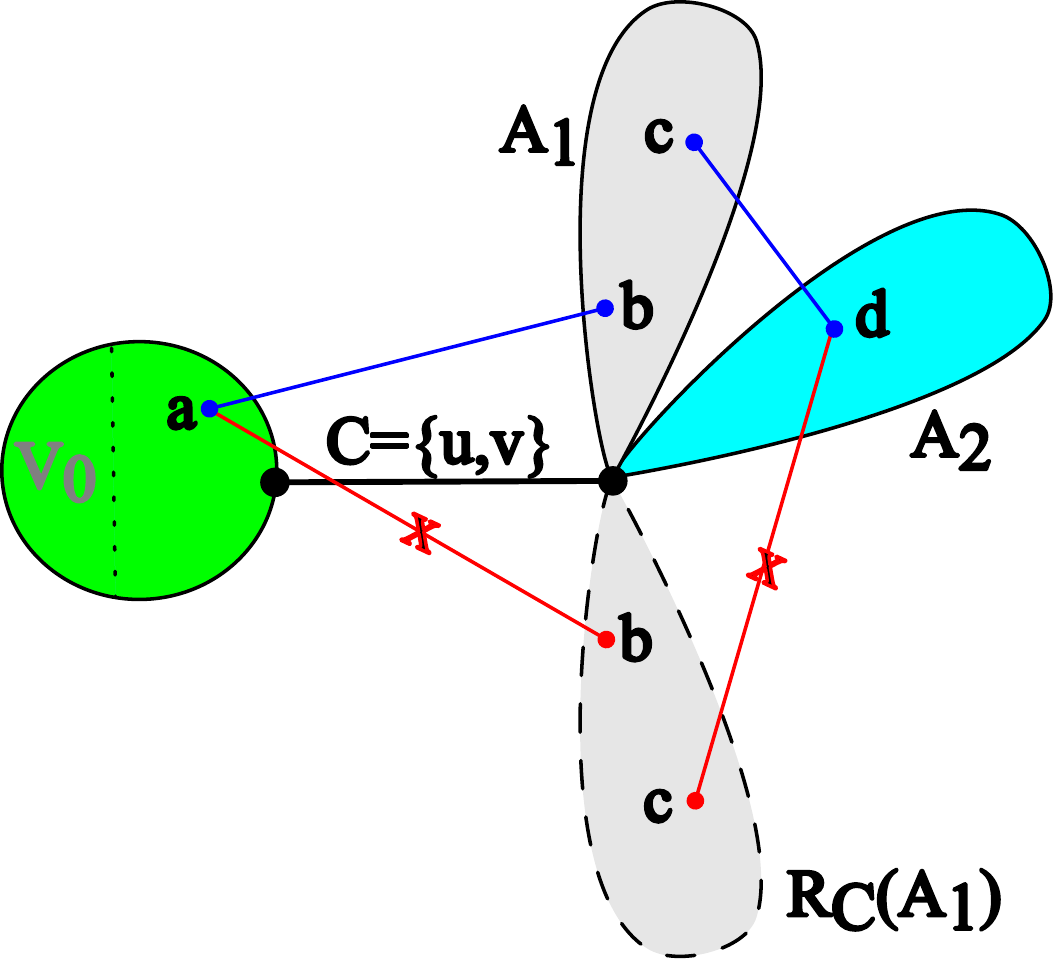}
  \caption{Component reflection across a separator \(C\). The blue segments
  represent distances before reflection and the red segments the corresponding
  tests after reflecting \(A_1\). The starred red segments are crossed pruning
  edges whose fixed endpoints lie outside \(C\); they are precisely the local
  violations recorded by the labeled matrix \(\mathsf V\).}
  \label{fig:component-reflection}
\end{figure}

For a lateration embedding \(x\), write
\(R_x^C\) for Euclidean reflection across
\(\operatorname{aff}_x(C)\). If \(S\subseteq V_P\setminus C\), let
\(R_S^C x\) denote the coordinate assignment obtained by applying
\(R_x^C\) to vertices of \(S\) and fixing all others.

\begin{lemma}[Structural preservation]
\label{lem:component-reflection}
Let \(S\) be a union of components of \(G_C\) with \(S\cap V_0=\varnothing\).
Then \(R_S^C x\) fixes the seed and preserves every edge in \(E_D[V_P]\).
\end{lemma}

\begin{proof}
The seed is fixed because \(S\cap V_0=\varnothing\). Consider
\(\{u,v\}\in E_D[V_P]\). If both endpoints lie in \(S\), the same isometry
is applied to both; if neither lies in \(S\), the edge is unchanged. If
exactly one endpoint lies in \(S\), the other cannot belong to
\(V_P\setminus C\), because \(S\) is a union of connected components of
\(G_C\). Hence the fixed endpoint lies in \(C\) and is fixed by the
reflection. In all cases the edge length is preserved.
\end{proof}

The next result identifies the corresponding branch-code change. It is the
reason for using oriented-volume signs rather than an arbitrary binary
labeling of the two candidates.

\begin{proposition}[Reflection-compatible labeling]
\label{prop:reflection-labeling}
Let \(C\in\mathcal C_P\), \(A\in\widehat G_C\), and let \(x\) be any
lateration embedding. Then
\[
  b_P(R_A^C x)\oplus b_P(x)=\mathbf1_{B_P\cap A}.
\]
\end{proposition}

\begin{proof}
Let \(v_j\) be the smallest-index vertex of \(A\). No predecessor of
\(v_j\) lies in \(A\), by minimality. A predecessor outside \(A\cup C\)
would be adjacent to \(v_j\) in \(G_C\), contradicting that \(A\) is a
component. Thus \(C_j\subseteq C\), and equality follows from
\(|C_j|=|C|=K\). In particular, every vertex of \(C\) precedes every vertex
of \(A\), so no vertex of \(C\) has a predecessor in \(A\).

Now fix \(v_i\in B_P\). If \(v_i\in A\), every predecessor in \(C_i\)
lies in \(A\cup C\); otherwise a structural edge would join \(A\) to a
different component of \(G_C\). Hence the simplex
\(C_i\cup\{v_i\}\) is transformed by the same orientation-reversing
reflection, with vertices in \(C\) fixed. Therefore
\(\Delta_i(R_A^C x)=-\Delta_i(x)\), and the bit flips.

If \(v_i\notin A\cup C\), no predecessor of \(v_i\) lies in \(A\), or
\(v_i\) would belong to the same component. Its local simplex is unchanged.
If \(v_i\in C\cap B_P\), the order argument above gives
\(C_i\cap A=\varnothing\), so its simplex is also unchanged. Exactly the bits
in \(B_P\cap A\) flip.
\end{proof}

\begin{definition}[Mirror masks and mask matrix]
\label{def:mask-matrix}
For \(C\in\mathcal C_P\) and \(A\in\widehat G_C\), define the labeled
mirror mask
\[
  m_A^C=\mathbf1_{B_P\cap A}\in\mathbb F_2^{B_P}.
\]
Let
\[
  \mathcal M=\{m_A^C:C\in\mathcal C_P,\ A\in\widehat G_C\}
\]
be an indexed family: equal vector values with different labels remain
distinct columns. The mask matrix is the linear map
\[
  \mathsf M:\mathbb F_2^{\mathcal M}\longrightarrow\mathbb F_2^{B_P}
\]
whose column indexed by \((C,A)\) is \(m_A^C\). Thus
\[
  \mathsf M\alpha=
  \bigoplus_{\substack{C\in\mathcal C_P,\ A\in\widehat G_C\\
                       \alpha_A^C=1}}m_A^C.
\]
\end{definition}

%% file: sections/violations.tex
\section{Labeled Violations and Global Soundness}
\label{sec:violations}

A component reflection can alter a pruning distance only when its support
contains exactly one endpoint. Even then, the distance is preserved when the
fixed endpoint belongs to the mirror clique. This motivates a violation test
that retains the mirror label.

\begin{definition}[Labeled violation matrix]
\label{def:violation-matrix}
For \(e=\{a,b\}\in E_P\), the labeled mask \(m_A^C\) \emph{violates} \(e\)
if
\[
  |e\cap\supp(m_A^C)|=1
  \qquad\text{and}\qquad
  e\setminus\supp(m_A^C)\not\subseteq C.
\]
Let
\[
  \mathcal R=
  \{(e,C):e\in E_P,\ C\in\mathcal C_P,
    \widehat G_C\ne\varnothing\}.
\]
The labeled violation matrix
\[
  \mathsf V:\mathbb F_2^{\mathcal M}\longrightarrow
  \mathbb F_2^{\mathcal R}
\]
has entries
\[
  \mathsf V_{(e,C),m_A^{C'}}=1
  \quad\Longleftrightarrow\quad
  C=C'\text{ and }m_A^{C'}\text{ violates }e.
\]
\end{definition}

Rows with different mirror labels are deliberately distinct. Two component
reflections across the same mirror can cancel a crossing parity; crossings
associated with different affine mirrors have no such common geometric
interpretation.

Define the kernel-compatible shift space
\[
  \mathcal K=\mathsf M(\ker\mathsf V)
  \subseteq\mathbb F_2^{B_P}.
\]

\begin{lemma}[Kernel compatibility]
\label{lem:kernel-compatibility}
Let \(\alpha\in\ker\mathsf V\). For each \(C\in\mathcal C_P\), set
\[
  S_C=\bigcup_{\substack{A\in\widehat G_C\\\alpha_A^C=1}}A.
\]
Every pruning edge crossed by \(S_C\) has its fixed endpoint in \(C\).
\end{lemma}

\begin{proof}
Suppose \(e=\{a,b\}\) is crossed by \(S_C\), with \(a\in S_C\) and
\(b\notin S_C\). Because the selected components are disjoint, the binary sum
of their masks contains exactly one endpoint of \(e\); hence an odd number of
selected \(C\)-masks cross \(e\). If \(b\notin C\), every such crossing is a
violation in row \((e,C)\), so
\((\mathsf V\alpha)_{(e,C)}=1\), contradicting
\(\alpha\in\ker\mathsf V\). Therefore \(b\in C\).
\end{proof}

To state soundness uniformly, fix a combinatorial template and let \(\Theta\)
be any common parameter domain on which Assumption~\ref{ass:sd} holds for all
branch codes. Write \(\Phi^\theta\) for the corresponding branch map and,
for \(e=\{a,b\}\in E_P\), define
\[
  Q_e^\theta(s)=
  \|\Phi_a^\theta(s)-\Phi_b^\theta(s)\|^2.
\]

\begin{theorem}[Global stabilizer soundness]
\label{thm:soundness}
For every \(k\in\mathcal K\), \(s\in\mathbb F_2^B\),
\(e\in E_P\), and \(\theta\in\Theta\),
\[
  Q_e^\theta(s\oplus k)=Q_e^\theta(s),
\]
where \(k\) is extended by zero on \(B_F\). Consequently, if
\(s^\ast\) is feasible, then
\[
  s_P^\ast\oplus\mathcal K\subseteq\mathcal X_P.
\]
\end{theorem}

\begin{proof}
Choose \(\alpha\in\ker\mathsf V\) with \(k=\mathsf M\alpha\). Starting from
\(x=\Phi^\theta(s)\), process the nonempty sets \(S_C\) one mirror label at
a time. At the current embedding, reflect \(S_C\) across
\(\operatorname{aff}_x(C)\). Lemma~\ref{lem:component-reflection} preserves
the lateration skeleton, and Lemma~\ref{lem:kernel-compatibility} ensures that
every crossed pruning edge has its fixed endpoint in \(C\). Thus every pruning
edge length is preserved: either neither endpoint moves, both move under the
same isometry, or the unique fixed endpoint lies on the mirror.

By Proposition~\ref{prop:reflection-labeling}, the accumulated branch-code
change is
\[
  \bigoplus_{\substack{C\in\mathcal C_P,\ A\in\widehat G_C\\
                       \alpha_A^C=1}}m_A^C
  =\mathsf M\alpha=k.
\]
The final constrained embedding therefore agrees with
\(\Phi^\theta(s\oplus k)\) at every pruning endpoint, and all pruning squared
distances equal their initial values. If \(s^\ast\) is feasible, those values
are the prescribed lengths, proving the inclusion.
\end{proof}

\begin{definition}[Relative feasible shifts]
\label{def:relative-shifts}
For a feasible reference code \(s^\ast\), define
\[
  \mathcal I_P(s^\ast)=s_P^\ast\oplus\mathcal X_P.
\]
Soundness is the inclusion \(\mathcal K\subseteq\mathcal I_P(s^\ast)\).
Completeness is the reverse inclusion.
\end{definition}

%% file: sections/results.tex
\section{Generic Completeness and the Rank Formula}
\label{sec:results}

We first isolate the linear-algebraic dimension that enters the count. For
linear maps represented over \(\mathbb F_2\), vertical concatenation is denoted
by \(\bigl[\begin{smallmatrix}\mathsf M\\\mathsf V\end{smallmatrix}\bigr]\).

\begin{lemma}[Dimension of the compatible shift space]
\label{lem:rank-formula}
The space \(\mathcal K=\mathsf M(\ker\mathsf V)\) satisfies
\[
  \dim\mathcal K=
  \rank_{\mathbb F_2}
  \begin{bmatrix}\mathsf M\\\mathsf V\end{bmatrix}
  -\rank_{\mathbb F_2}(\mathsf V).
\]
\end{lemma}

\begin{proof}
Restrict \(\mathsf M\) to \(\ker\mathsf V\). Rank--nullity gives
\[
 \dim\mathsf M(\ker\mathsf V)
 =\dim\ker\mathsf V-\dim(\ker\mathsf M\cap\ker\mathsf V).
\]
The two terms on the right are, respectively,
\(|\mathcal M|-\rank\mathsf V\) and
\(|\mathcal M|-\rank[\begin{smallmatrix}\mathsf M\\\mathsf V\end{smallmatrix}]\).
Subtracting proves the claim.
\end{proof}

We now state the genericity condition needed for completeness. Assume
\(E_P\ne\varnothing\), put
\[
  L=(V_P,E_0\cup E_D[V_P]),
  \qquad H=(V_P,E(L)\cup E_P),
\]
and regard both graphs as graphs in dimension \(K\).

\begin{assumption}[Generic feasible framework (GP)]
\label{ass:gp}
The constrained graph \(H\) has a feasible realization that is congruent to
a generic framework in \(\mathbb R^K\).
\end{assumption}

Genericity in Assumption~\ref{ass:gp} is imposed before seed normalization.
Indeed, prescribed seed coordinates cannot themselves be algebraically
generic. An isometry may subsequently place the \(K\) affinely independent
seed points at their prescribed positions without changing edge lengths.

\begin{theorem}[Dimension-uniform generic completeness]
\label{thm:generic-completeness}
Let \(K\ge1\). Under the skeletal CDDGP convention and
Assumptions~\ref{ass:sd} and~\ref{ass:gp}, every feasible reference code
\(s^\ast\) satisfies
\[
  \mathcal I_P(s^\ast)=\mathcal K,
  \qquad
  \mathcal X_P=s_P^\ast\oplus\mathcal K.
\]
If \(E_P=\varnothing\), the same conclusion holds with
\(B_P=\varnothing\) and \(\mathcal K=\{0\}\), without Assumption~\ref{ass:gp}.
\end{theorem}

\begin{proof}
The case \(E_P=\varnothing\) follows directly from the conventions in
Section~\ref{sec:subproblems}; assume henceforth that \(E_P\ne\varnothing\).
Predecessor closure lists the vertices of \(B_P\) in the inherited order with
all predecessors already present. The first such vertex attaches to \(V_0\),
and each later vertex attaches to the existing \(K\)-clique \(C_i\).
Consequently, \(L\) is a \(K\)-tree and hence is \(K\)-connected and chordal.
It is therefore \(K\)-joined by \cite[Theorem~4.12]{garamvolgyi2024partial};
adding the pruning edges preserves this property by
\cite[Theorem~4.7]{garamvolgyi2024partial}. Thus \(H\) is \(K\)-joined.

Let \(C\) be a \(K\)-vertex separator of \(H\). Since \(V_0\) is a clique,
at most one component of \(H-C\) contains a seed vertex outside \(C\), so
\(H-C\) has a seed-free component \(Q\). If \(v_i\) is the smallest-index
vertex of \(Q\), none of its predecessors lies in \(Q\), and none lies in
another component of \(H-C\), because it is structurally adjacent to \(v_i\).
Hence \(C_i\subseteq C\), and cardinality gives \(C=C_i\in\mathcal C_P\).

Every realization equivalent to a generic realization of a \(K\)-joined
graph is congruent to one obtained by a reduced sequence of partial
\(K\)-reflections \cite[Theorem~4.5]{garamvolgyi2024partial}. By replacing a
fragment with its complementary fragment when necessary
\cite[Lemmas~3.5--3.6]{garamvolgyi2024partial}, each reflected side \(Y\) may
be chosen seed-free, with \(N_H(Y)=C\). Since no structural edge joins \(Y\)
to \(V_P\setminus(Y\cup C)\), the set \(Y\) is a union of relevant components
of \(G_C\). Let \(\alpha_Y\) select precisely those labeled columns. Then
\[
  \mathsf M\alpha_Y=\mathbf1_Y.
\]
Rows of \(\mathsf V\) with a different mirror label vanish. For label \(C\),
a pruning edge whose endpoints lie in two selected components contributes two
violations, which cancel over \(\mathbb F_2\); an edge within one component
contributes none. Hence a nonzero row would require a pruning edge from \(Y\)
to \(V_P\setminus(Y\cup C)\), contrary to \(N_H(Y)=C\). Thus
\(\mathsf V\alpha_Y=0\), and every partial-reflection shift belongs to
\(\mathcal K\).

There is one possible residual congruence. Selecting all components of
\(G_{V_0}\) gives \(\mathsf M\alpha_0=\mathbf1_{B_P}\). A pruning edge within
one component is not crossed; an edge between two components contributes two
violations, which cancel; and an edge incident with \(V_0\) has its fixed
endpoint on the mirror. Hence \(\mathsf V\alpha_0=0\), so the global
reflection across \(\operatorname{aff}(V_0)\) also has its shift in
\(\mathcal K\).

Let \(p\) be the generic feasible representative supplied by GP and normalize
its seed by an isometry. Any other seed-fixed feasible constrained embedding
is equivalent to \(p\), so the preceding characterization expresses it by a
sequence of partial reflections followed by a congruence. The sequence fixes
the seed and has total branch shift in \(\mathcal K\). A Euclidean congruence
fixing the \(K\) affinely independent seed points is either the identity or
reflection across their affine hull; the latter contributes
\(\mathbf1_{B_P}\in\mathcal K\). Therefore
\(\mathcal X_P\subseteq s_P^g\oplus\mathcal K\). Global soundness
(Theorem~\ref{thm:soundness}) gives the reverse inclusion, and hence
\(\mathcal X_P=s_P^g\oplus\mathcal K\). For any feasible reference
\(s^\ast\), this equality implies
\(s_P^\ast\oplus s_P^g\in\mathcal K\), so
\[
  \mathcal X_P=s_P^\ast\oplus\mathcal K,
  \qquad
  \mathcal I_P(s^\ast)=\mathcal K.
\]
This completes the proof.
\end{proof}

\begin{theorem}[Rank-count formula]
\label{thm:rank-count}
Under the hypotheses of Theorem~\ref{thm:generic-completeness}, the complete
set of feasible branch codes is
\[
  \mathcal X\cong
  (s_P^\ast\oplus\mathcal K)\times\mathbb F_2^{B_F},
\]
and its cardinality is
\[
  |\mathcal X|=
  2^{|B_F|+
  \rank_{\mathbb F_2}
  \left[\begin{smallmatrix}\mathsf M\\\mathsf V\end{smallmatrix}\right]
  -\rank_{\mathbb F_2}(\mathsf V)}.
\]
In particular, when \(E_P=\varnothing\), this reduces to
\(|\mathcal X|=2^{n-K}\).
\end{theorem}

\begin{proof}
Combine Theorem~\ref{thm:generic-completeness},
Proposition~\ref{prop:free-bits}, and Lemma~\ref{lem:rank-formula}.
\end{proof}

%% file: sections/example.tex
\section{Worked Example and Computational Check}
\label{sec:example}

Consider a planar instance (\(K=2\)) with seed
\(V_0=\{v_1,v_2\}\), predecessor sets
\[
 C_3=C_7=\{v_1,v_2\},\quad
 C_4=C_5=\{v_1,v_3\},\quad
 C_6=\{v_1,v_4\},
\]
and pruning edges
\[
 E_P=\bigl\{\{v_3,v_6\},\{v_4,v_5\}\bigr\}.
\]
The predecessor closure gives
\(V_P=\{v_1,\ldots,v_6\}\),
\(B_P=\{v_3,v_4,v_5,v_6\}\), and \(B_F=\{v_7\}\).

Write \(C_0=\{v_1,v_2\}\), \(C_1=\{v_1,v_3\}\), and
\(C_2=\{v_1,v_4\}\). The relevant components, listed with their mirror
labels, are
\[
 (C_0,\{v_3,v_4,v_5,v_6\}),\quad
 (C_1,\{v_4,v_6\}),\quad
 (C_1,\{v_5\}),\quad
 (C_2,\{v_6\}).
\]
Ordering rows by \((v_3,v_4,v_5,v_6)\) and columns as displayed above gives
\[
 \mathsf M=
 \begin{bmatrix}
 1&0&0&0\\
 1&1&0&0\\
 1&0&1&0\\
 1&1&0&1
 \end{bmatrix}.
\]
Only two labeled violation rows are nonzero: the row
\((\{v_4,v_5\},C_1)\) and the row \((\{v_3,v_6\},C_2)\). Thus
\[
 \mathsf V=
 \begin{bmatrix}
 0&1&1&0\\
 0&0&0&1
 \end{bmatrix}.
\]
It follows that \(\rank\mathsf V=2\) and
\(\rank[\begin{smallmatrix}\mathsf M\\\mathsf V\end{smallmatrix}]=4\).
The compatible constrained shifts are
\[
 \mathcal K=\left\{
 (0,0,0,0)^\mathsf T,
 (1,1,1,1)^\mathsf T,
 (0,1,1,1)^\mathsf T,
 (1,0,0,0)^\mathsf T
 \right\}.
\]
Therefore \(\dim\mathcal K=2\), and the free decision at \(v_7\) doubles
the four constrained codes:
\[
 |\mathcal X|=2^{1+4-2}=8.
\]

%% file: sections/conclusion.tex
\section{Conclusion}
\label{sec:conclusion}

We derived an exact, dimension-uniform count for generic feasible skeletal
CDDGP instances. The construction separates branch decisions outside the
predecessor closure of pruning endpoints, which remain free, from constrained
decisions governed by partial reflections. Relevant components of the
lateration skeleton provide binary masks, while the labeled violation matrix
records the pruning edges incompatible with each mirror. Soundness holds on
the full strict-discretization parameter domain; generic completeness follows
from the partial-reflection description of equivalent realizations of
\(K\)-joined graphs. Together they identify the constrained feasible codes
with one affine translate of \(\mathsf M(\ker\mathsf V)\).

The resulting formula replaces exhaustive traversal of the binary lateration
tree by graph operations and ranks over \(\mathbb F_2\). Its scope is also
explicit: the skeletal clique condition separates lateration from pruning,
SD excludes collapsed branches, and GP rules out nongeneric coincidences not
explained by graph separators. These hypotheses make the formula suitable as
a certifiable counting step whenever a generic feasible representative is
available.

Several algorithmic questions remain. Sparse construction and elimination for
the labeled matrices should be evaluated on large molecular and localization
instances. It would also be useful to characterize nongeneric exceptional
instances, to test genericity without symbolic coordinates, and to determine
which weaker structural assumptions retain an affine feasible-code space.

%% file: references.bib
@article{lavor2012discretizable,
  title={The discretizable molecular distance geometry problem},
  author={Lavor, Carlile and Liberti, Leo and Maculan, Nelson and Mucherino, Antonio},
  journal={Computational Optimization and Applications},
  volume={52},
  number={1},
  pages={115--146},
  year={2012},
  doi={10.1007/s10589-011-9402-6},
  publisher={Springer}
}

@article{abud2024impossible,
  title={An impossible combinatorial counting method in distance geometry},
  author={Abud, Germano and Alencar, Jorge and Lavor, Carlile and Liberti, Leo and Mucherino, Antonio},
  journal={Discrete Applied Mathematics},
  volume={354},
  pages={83--93},
  year={2024},
  doi={10.1016/j.dam.2024.02.018},
  publisher={Elsevier}
}

@article{mucherino2012exploiting,
  title={Exploiting symmetry properties of the discretizable molecular distance geometry problem},
  author={Mucherino, Antonio and Lavor, Carlile and Liberti, Leo},
  journal={Journal of Bioinformatics and Computational Biology},
  volume={10},
  number={03},
  pages={1242009},
  year={2012},
  doi={10.1142/S0219720012420097},
  publisher={World Scientific}
}

@inproceedings{liberti2011number,
  title={On the number of solutions of the discretizable molecular distance geometry problem},
  author={Liberti, Leo and Masson, Beno{\^\i}t and Lee, Jon and Lavor, Carlile and Mucherino, Antonio},
  booktitle={International Conference on Combinatorial Optimization and Applications},
  pages={322--342},
  year={2011},
  doi={10.1007/978-3-642-22616-8_26},
  organization={Springer}
}

@article{mucherino2012discretizable,
  title={The discretizable distance geometry problem},
  author={Mucherino, Antonio and Lavor, Carlile and Liberti, Leo},
  journal={Optimization Letters},
  volume={6},
  number={8},
  pages={1671--1686},
  year={2012},
  doi={10.1007/s11590-011-0358-3},
  publisher={Springer}
}

@article{goncalves2021new,
  title={A new algorithm for the {$K$}-{DMDGP} subclass of distance geometry problems with exact distances},
  author={Gon{\c{c}}alves, Douglas S and Lavor, Carlile and Liberti, Leo and Souza, Michael},
  journal={Algorithmica},
  volume={83},
  number={8},
  pages={2400--2426},
  year={2021},
  doi={10.1007/s00453-021-00835-6},
  publisher={Springer}
}

@article{cassioli2015discretization,
  title={Discretization vertex orders in distance geometry},
  author={Cassioli, Andrea and G{\"u}nl{\"u}k, Oktay and Lavor, Carlile and Liberti, Leo},
  journal={Discrete Applied Mathematics},
  volume={197},
  pages={27--41},
  year={2015},
  doi={10.1016/j.dam.2014.08.035},
  publisher={Elsevier}
}

@article{liberti2014number,
  title={On the number of realizations of certain Henneberg graphs arising in protein conformation},
  author={Liberti, Leo and Masson, Beno{\^\i}t and Lee, Jon and Lavor, Carlile and Mucherino, Antonio},
  journal={Discrete Applied Mathematics},
  volume={165},
  pages={213--232},
  year={2014},
  doi={10.1016/j.dam.2013.01.020},
  publisher={Elsevier}
}

@article{liberti2008branch,
  title={A branch-and-prune algorithm for the molecular distance geometry problem},
  author={Liberti, Leo and Lavor, Carlile and Maculan, Nelson},
  journal={International Transactions in Operational Research},
  volume={15},
  number={1},
  pages={1--17},
  year={2008},
  doi={10.1111/j.1475-3995.2007.00622.x},
  publisher={Wiley Online Library}
}

@incollection{mucherino2020analysis,
  title={An analysis on the degrees of freedom of binary representations for solutions to discretizable distance geometry problems},
  author={Mucherino, Antonio},
  booktitle={Recent Advances in Computational Optimization},
  pages={251--255},
  year={2021},
  doi={10.1007/978-3-030-82397-9},
  publisher={Springer}
}

@article{abud2018k,
  title={The {K}-discretization and {K}-incident graphs for discretizable Distance Geometry},
  author={Abud, Germano and Alencar, Jorge and Lavor, Carlile and Liberti, Leo and Mucherino, Antonio},
  journal={Optimization Letters},
  volume={14},
  number={2},
  pages={1--14},
  year={2018},
  doi={10.1007/s11590-018-1294-2}
}

@inproceedings{liberti2013counting,
  title={Counting the number of solutions of kDMDGP instances},
  author={Liberti, Leo and Lavor, Carlile and Alencar, Jorge and Abud, Germano},
  booktitle={International Conference on Geometric Science of Information},
  pages={224--230},
  year={2013},
  doi={10.1007/978-3-642-40020-9_23},
  organization={Springer}
}

@inproceedings{yemini1978positioning,
  title={The positioning problem—a draft of an intermediate summary},
  author={Yemini, Yechiam},
  booktitle={Proceedings of the Conference on Distributed Sensor Networks},
  pages={137--145},
  year={1978},
  organization={sn}
}

@article{liberti2020distance,
  title={Distance geometry and data science},
  author={Liberti, Leo},
  journal={Top},
  volume={28},
  number={2},
  pages={271--339},
  year={2020},
  doi={10.1007/s11750-020-00563-0},
  publisher={Springer}
}

@article{garamvolgyi2024partial,
  author={Garamv{\"o}lgyi, D{\'a}niel and Jord{\'a}n, Tibor},
  title={Partial Reflections and Globally Linked Pairs in Rigid Graphs},
  journal={SIAM Journal on Discrete Mathematics},
  volume={38},
  number={3},
  pages={2005--2040},
  year={2024},
  doi={10.1137/23M157065X}
}
